\documentclass[12pt]{amsart}
\usepackage{amssymb}
\usepackage{amsfonts}
\usepackage{graphicx}
\usepackage{amsmath}
\usepackage{xcolor}

\newtheorem{theorem}{Theorem}

\newtheorem{definition}[theorem]{Definition}

\newtheorem{lemma}[theorem]{Lemma}

\newcommand{\N}{\mathbb{N}}

\newcommand{\R}{\mathbb{R}}

\begin{document}
\title[]{The Dirichlet problem with entire data for non-hyperbolic quadratic hypersurfaces}
\author{J. M. Aldaz and H. Render}
\address{H. Render: School of Mathematical Sciences, University College
Dublin, Dublin 4, Ireland.}
\email{hermann.render@ucd.ie}
\address{J.M. Aldaz: Instituto de Ciencias Matem\'aticas (CSIC-UAM-UC3M-UCM)
and Departamento de Matem\'aticas, Universidad Aut\'onoma de Madrid,
Cantoblanco 28049, Madrid, Spain.}
\email{jesus.munarriz@uam.es}
\email{jesus.munarriz@icmat.es}
\thanks{2020 Mathematics Subject Classification: \emph{Primary: 35A20}, 
\emph{Secondary: 35A01}}
\thanks{Key words and phrases: \emph{Dirichlet problem, entire harmonic function, nonhyperbolic quadratic surface}}
\thanks{The first named author was partially supported by Grant PID2019-106870GB-I00 of the MICINN of Spain,  by ICMAT Severo Ochoa project
CEX2019-000904-S (MICINN), and by
the Madrid Government (Comunidad de Madrid - Spain)  V PRICIT (Regional Programme of Research and Technological Innovation), 2022-2024.}

\begin{abstract}
We show that for all homogeneous polynomials $
	f_{m}$ of degree $m$, in $d$ variables, 
	and each $j = 1, \dots , d$, we have
	\begin{equation*}
		\left\langle x_{j}^{2}f_{m},f_{m}\right\rangle _{L^{2}\left( \mathbb{S}%
			^{d-1}\right) }
			\geq
			\frac{\pi ^{2}}{4\left( m+ 2 d + 1 \right)^{2}}
			\left
			\langle
		f_{m},f_{m}\right\rangle _{L^{2}\left( \mathbb{S}^{d-1}\right) }.
		\end{equation*}
This result is used to establish the existence of entire harmonic solutions
of the Dirichlet problem,  when the data are given
by entire functions of order sufficiently low 
on nonhyperbolic quadratic hypersurfaces. 
\end{abstract}

\maketitle

\section{Introduction}

The Dirichlet problem has been solved for very general domains of $\R^d$. S. J. Gardiner proved that solutions exist, and
not just for continuous boundary data on unbounded domains, but  even for 
locally bounded Borel data (cf. the characterization given in  
 \cite[Theorem 1]{Gard93} and the comments following it). A different question that naturally arises
 is whether harmonic solutions can be extended beyond the boundary when the
 data is sufficiently regular, perhaps yielding global solutions. Additional information on these matters can be found in reference
 \cite{KL}.
 
 One way of ensuring the existence of global solutions is to prove
 that if the data is given by a polynomial, or more generally, by an entire function, the solution
 is also entire. Following a line of research started in \cite{Shap89},  D. Khavinson and H. S. Shapiro showed that for every entire function $f$ on $\mathbb{C}^{d}$, the solution $h$
of the Dirichlet problem for the ellipsoid with data function $f$ (restricted to
the boundary) has a harmonic continuation to $\mathbb{R}^{d}$, and hence it extends to an entire function on $\mathbb{C}^{d}$, see \cite[Theorem 1]{KhSh92}. They also asked whether this property characterized ellipsoids,
cf. \cite[Page 460]{KhSh92}, a problem that  is still open in full generality, though
under some additional assumptions a positive answer is known (see \cite[Theorem 27]{Rend08}). An equivalent version of the Khavinson-Shapiro problem appears in \cite{Rend16}.
 
 The aim of the present paper is to show that results related to the Khavinson-Shapiro Theorem hold for Dirichlet
problems on some specific, not necessarily bounded domains in $\mathbb{R}^{d}$, namely, on nonhyperbolic quadratic hypersurfaces.  The Dirichlet problem for polynomial data on such hypersurfaces 
was studied in \cite{AGV04}, where the authors present an algorithm to compute polynomial harmonic solutions.
Next we recall  \cite[Definition 4]{AGV04}. 

\begin{definition}   \label{D:nonhyperbolic} Let  $a_{1},\dots , a_d$, $b_1, \ldots, b_d$ and $c$,  be real numbers.
A \textit{nonhyperbolic} quadratic polynomial $q$ has the form
\[
q(x) := \sum_{j = 1}^d a_j^2 x_j^2 + \sum_{j=1}^d b_jx_j + c,
\]
where at least one $a_j \ne 0$. A nonhyperbolic quadratic hypersurface is the zero set $\{q = 0\}$.
\end{definition}

Writing $q = P_2 + P_1 + P_0$, where $P_i$ denotes the homogeneous part of $q$ of degree $i$, we set 
$\beta =1$ if $P_1 \ne 0$ and $\beta = 0$ otherwise. Then we have

\begin{theorem} \label{arbdimParCyl}
Let  $d\geq 2$, and let $q$  be a nonhyperbolic quadratic polynomial. If $f$ is an entire function of order $\rho < 1 - \beta/2$, then then there is a harmonic entire function of order bounded by $\rho$, which coincides with $f$ 
on $\{q = 0\}$.
\end{theorem}

As special cases we mention paraboloids, where $q(x) =  a_2^2 x_{2}^{2}+\cdots +a_d^2 x_{d}^{2} - x_1$
and  $\rho < 1/2$, ellipsoidal cylinders, where $q(x) = a_2^2 x_{2}^{2}+\cdots +a_d^2 x_{d}^{2} - 1$
and  $\rho < 1$, and slabs, where $q(x) = a_d^2 x_{d}^{2} - 1$
and  $\rho < 1$. The case of the ellipsoidal cylinders had already been treated in  \cite{KLR17} by potential-theoretic methods, under the same hypothesis $\rho < 1$, but without reaching the conclusion that the
 harmonic solution has finite order bounded by $\rho$.

 The question whether a version of Theorem \ref{arbdimParCyl}   holds without any assumptions  on the order of  $f$, is both open and interesting.  For slabs, the existence of harmonic entire solutions
 is obtained in \cite{KLR17b}, under no assumptions on the order of the entire data functions. But there the argument depends on the Schwarz  reflection principle for harmonic functions vanishing on  hyperplanes, so this technique is not available for general quadratic surfaces.  In passing, we mention that when  $d=2$, harmonic entire functions vanishing on a parabola  are characterized in  \cite[p. 427]{FNS66}; it seems plausible that in this case the answer might be affirmative.

 Finally, in the (bounded) case of ellipsoids, we have already mentioned the Khavinson-Shapiro Theorem for arbitrary entire functions. They also note that if the data function is of exponential type, so is the harmonic solution (cf. \cite[P. 459 Remarks]{KhSh92}). In \cite[Theorem 1]{Armi04} D. Armitage shows that the order of the solution is bounded by the order of the data function, and if the order is the same, then the type of the solution is bounded by the type of the data function.
 Thus, for ellipsoids Theorem \ref{arbdimParCyl} only yields strictly weaker results: since via a rotation any ellipsoid can be expressed as $q(x) = a_1^2 x_{1}^{2}+\cdots +a_d^2 x_{d}^{2} - 1= 0$,  we take $\beta = 0$ and thus, $\rho < 1$.
 
 The results presented here are obtained by generalizing to arbitrary dimensions Theorems 3, 4 and 5 from
 \cite{RendAl22}, proven there only for the case $d =2$. Theorem \ref{arbdimParCyl} is obtained via the following estimate, together with \cite[Theorem 1]{RendAl22}).

 \begin{theorem} \label{arbdimB}
 	For all homogeneous polynomials $
 	f_{m}$ of degree $m$ in $d$ variables, 
 	and each $j = 1, \dots , d$, we have
 	\begin{equation*}
 		\left\langle x_{j}^{2}f_{m},f_{m}\right\rangle _{L^{2}\left( \mathbb{S}
 			^{d-1}\right) }
 		\geq
 		\frac{\pi ^{2}}{4\left( m + 2 d + 1 \right)^{2}}
 		\left
 		\langle
 		f_{m},f_{m}\right\rangle _{L^{2}\left( \mathbb{S}^{d-1}\right) }.
 	\end{equation*}
 \end{theorem}

\section{Jacobi polynomials}

In the following we need the Jacobi polynomials $P_{n}^{\left( \alpha
,\alpha \right) }\left( x\right) $ of degree $n\ge 0$ and parameter  \color{red} $\alpha > -1,$ \color{black}
see \cite[2.4.1, p. 29]{Szeg39}. Here $x\in\R$. For $n \ge 1$ the Jacobi polynomials $P_{n}^{\left( \alpha ,\alpha
\right) }\left( x\right) $ satisfy the recurrence relation 
\begin{equation*}
xP_{n}^{\left( \alpha ,\alpha \right) }\left( x\right) =a_{n}\left( \alpha
\right) P_{n+1}^{\left( \alpha ,\alpha \right) }+g_{n}\left( \alpha \right)
P_{n-1}^{\left( \alpha ,\alpha \right) }\left( x\right)
\end{equation*}
with coefficients $a_{n}\left( \alpha \right) $ and $g_{n}\left( \alpha
\right) $ defined by 
\begin{eqnarray*}
a_{n}\left( \alpha \right) &=&\frac{\left( n+1\right) \left( n+2\alpha
+1\right) }{\left( 2n+2\alpha +1\right) \left( n+\alpha +1\right) }, \\
g_{n}\left( \alpha \right) &=&\frac{n+\alpha }{2n+2\alpha +1},
\end{eqnarray*}
and initial conditions $P_{0}^{\left( \alpha ,\alpha \right) }\left(
x\right) =1$ and $P_{1}^{\left( \alpha ,\alpha \right) }\left( x\right)
=\left( \alpha +1\right) x.$ It is easy to see that for $n \ge 2$,
\begin{equation}
x^{2}P_{n}^{\left( \alpha ,\alpha \right) }\left( x\right) =\widetilde{a}%
_{n}\left( \alpha \right) P_{n+2}^{\left( \alpha ,\alpha \right) }\left(
x\right) +\widetilde{b}_{n}\left( \alpha \right) P_{n}^{\left( \alpha
,\alpha \right) }\left( x\right) +\widetilde{g}_{n}\left( \alpha \right)
P_{n-2}^{\left( \alpha ,\alpha \right) }\left( x\right),  \label{eqrec2}
\end{equation}%
where $\widetilde{a}_{n}=a_{n}a_{n+1}$, $\widetilde{g}_{n}=g_{n}g_{n-1}$,
and $\widetilde{b}_{n}=a_{n}g_{n+1}+g_{n}a_{n-1}.$
Write $0 \le y := x^2$ and $u := 2n$. Note that since 
$P_{2 n}^{\left( \alpha ,\alpha \right) }\left( x\right)$ is even (cf. \cite[(4.1.3)]{Szeg39}) the function 
$r_u(y) := P_{2 n}^{\left( \alpha ,\alpha \right) }\left( \sqrt{y}\right)$ is a polynomial on $[0, \infty)$, and furthermore,
\begin{equation}
 y \ r_{u} \left( y \right) =
 \widetilde{a}_{u}\left( \alpha \right) r_{u+1} \left(
	y \right) +\widetilde{b}_{u}\left( \alpha \right)  
r_{u} \left( y  \right)  +\widetilde{g}_{u}\left( \alpha \right)
 r_{u - 1} \left(
	 y  \right).  \label{eqrec3}
\end{equation}

Next we need an estimate of the first positive root of $P_{2n}^{\left(
\alpha ,\alpha \right) }\left( x\right) .$ Using the work of \'{A}
. Elbert and D. Siafarikas (see \cite{ElSi99}) we can derive quickly the following result:

\begin{theorem} \label{Elbert}
Let $x_{2n,1}\left( \alpha \right) $ be the first positive zero of $
P_{2n}^{\left( \alpha ,\alpha \right) }\left( x\right) $, where $n\geq 3$ and $\alpha \geq -1/2.$ Then 
\begin{equation*}
x_{2n,1}\left( \alpha \right) \geq \frac{\pi }{4\sqrt{\left( \alpha +\frac{1
}{2}+n\right) \left( n+2\right) }}.
\end{equation*}

\end{theorem}

\begin{proof}
	For $\alpha > - 1/2$ it is well known that $P_{2n}^{\left( \alpha ,\alpha \right) }\left(
x\right) $ is equal to a multiple of the Gegenbauer (or ultraspherical)
polynomial $C_{2n}^{\lambda }\left( x\right) $ of degree $2n$ with parameter 
$\lambda =\alpha +1/2.$ Denote the first positive zero of $C_{2n}^{\lambda }$
by $y_{2n,1}\left( \lambda \right) $, so $x_{2n,1}\left( \alpha \right)
=y_{2n,1}\left( \alpha +\frac{1}{2}\right) .$ In \cite[p. 32]{ElSi99}, the authors show
 that on $\left( 0,\infty \right)$, the function $f_{2n}$ defined by 
\begin{equation*}
f_{2n}\left( \lambda \right) :=\sqrt{\lambda +\frac{8n^{2}+1}{8n+2}}
 \ y_{2n,1}\left( \lambda \right)
 \ge
 \sqrt{\frac{8n^{2}+1}{8n+2}}
 \ y_{2n,1}\left( \lambda \right),
\end{equation*}
is increasing. 

By \cite[Theorem 6.21.1]{Szeg39}, as $\lambda \downarrow 0$, so $\alpha \downarrow -1/2$, 
we have $x_{2n,1}\left( \alpha \right) \uparrow x_{2n,1}\left( -1/2 \right)$. Since
$f_{2n}\left( \lambda \right)$ is increasing, we conclude that for every $\lambda > 0$,
\begin{equation*}
f_{2n}\left( \lambda \right) \geq 
\lim_{\lambda \downarrow 0}f_{2n}\left( \lambda \right) 
\ge 
\sqrt{\frac{8n^{2}+1
}{8n+2}} \ x_{2n,1}\left( -1/2 \right).
\end{equation*}
 Now by \cite[pg. 60]{Szeg39}, we have $P_{2n}^{\left( -1/2 , -1/2 \right) }\left(
 x\right) = c_{2n} T_{2n}\left( x\right) $, where $c_{2n}$ is a constant   depending only on $n$ and $T_{2n}\left( x\right) $  is the corresponding Chebyshev
polynomial  of the first kind. Since the zeros of $T_{2n}$
are given by $\cos \left( \frac{2k-1}{2n}\frac{\pi }{2}\right) $ for $%
k=1, \dots , 2n,$ we conclude that 
\begin{eqnarray*}
 x_{2n,1}\left( -1/2 \right) &=&\cos \left( \frac{2n-1}{2n}\frac{\pi }{2}\right)
=\sin \left( \frac{\pi }{2}-\frac{2n-1}{2n}\frac{\pi }{2}\right) \\
&=&\sin \frac{\pi }{4n}\geq \frac{\pi }{4n+2},
\end{eqnarray*}
where we have used that $\cos x=\sin \left( \frac{\pi }{2}-x\right)$,  and  for the last
inequality, \cite[Lemma 14]{RendAl22}.  Since 
\begin{equation*}
\sqrt{\alpha +\frac{1}{2}+\frac{8n^{2}+1}{8n+2}} \ y_{2n,1}\left( \alpha +\frac{1
}{2}\right) 
\ge
\sqrt{\frac{8n^{2}+1
	}{8n+2}} \ x_{2n,1}\left( -1/2 \right),
\end{equation*}
 it follows that 
\begin{equation*}
\left( \alpha +\frac{1}{2}+\frac{8n^{2}+1}{8n+2}\right) \left(
x_{2n,1}\left( \alpha \right) \right) ^{2}\geq \frac{8n^{2}+1}{8n+2}\frac{
\pi ^{2}}{\left( 4n+2\right) ^{2}}.
\end{equation*}
From 
$$
n\geq \frac{8n^{2}+1}{8n+2} \mbox{ \ and  \  } \frac{8n^{2}+1}{8n+2}\frac{1}{
\left( 4n+2\right) ^{2}}\geq \frac{1}{16\left( n+2\right) }
$$
 we obtain the
estimate 
\begin{equation*}
\left( \alpha +\frac{1}{2}+n\right) \left( x_{2n,1}\left( \alpha \right)
\right) ^{2}\geq \frac{1}{16\left( n+2\right) }\pi ^{2}.
\end{equation*}
\end{proof}

\section{Proof of Theorem \ref{arbdimB}}

Denote the unit sphere by 
$
\mathbb{S}^{d-1}=\left\{ \theta \in \mathbb{R}^{d}:\left\vert \theta
\right\vert =1\right\},
$
and its surface area measure by $d\theta $. The  $L^{2}(\mathbb{S}^{d-1})$  inner product and its associated norm, are defined in the usual way:
\begin{equation}
\left\langle f,g\right\rangle _{L^{2}(\mathbb{S}^{d-1})}:=\int_{\mathbb{S}%
^{d-1}}f\left( \theta \right) \overline{g\left( \theta \right) }d\theta 
\text{ and }\left\Vert f\right\Vert _{L^{2}(\mathbb{S}^{d-1})}=\sqrt{%
\left\langle f,f\right\rangle _{L^{2}(\mathbb{S}^{d-1})}}.
\label{eqprodspher}
\end{equation}

We shall need some results from the theory of spherical harmonics. It is well known that one can
construct \emph{orthogonal} spherical harmonics by induction over the
dimension $d$. In order to emphasize the dependence on $d$, we shall denote an orthogonal basis of the space $\mathcal{H}%
_{k}\left( \mathbb{R}^{d}\right)$ of all homogeneous harmonic polynomials of degree $k$ in $d$
variables, by 
\begin{equation*}
Y_{k,l}^{d}\left( x_{1},\dots,x_{d}\right) \text{ for }l=1,\dots,a_{k}^{\left(
d\right) },
\end{equation*}%
where we have written $a_{k}^{\left( d\right) }:=\dim \mathcal{H}_{k}\left( 
\mathbb{R}^{d}\right) .$ We also use  $x =\left(
x_{1},\dots ,x_{d}\right) $ and $\left\vert x\right\vert =\sqrt{
x_{1}^{2}+ \cdots +x_{d}^{2}}.$
The following result can be found, for instance, in \cite[p.
461]{AAR99} or in \cite{Frya91}.

\begin{theorem}
\label{ThmSphRec} Let $d\geq 2$, let $k \in \N$, and  for each $s=0,1,2,\dots,k$, let 
$$\{
Y_{s,l}^{d-1}: l=1,\dots, a_{s}^{\left( d-1\right) }\}
$$ be an orthogonal basis of $
\mathcal{H}_{s}\left( \mathbb{R}^{d-1}\right)$.
Then the system of polynomials given by
\begin{equation}
Y_{k,\left( s,l\right) }^{d}\left( x\right) 
:=
Y_{s,l}^{d-1}\left(
x_{1},\dots,x_{d-1}\right) \cdot \left\vert x\right\vert ^{k-s}P_{k-s}^{\left(
s+\left( d-3\right) /2, s+\left( d-3\right) /2\right) }\left( \frac{x_{d}}{%
\left\vert x\right\vert }\right)   \label{eqintFF},
\end{equation}
where $s=0,1,2,\dots ,k$ and $l=1,\dots,a_{s}^{(d-1)}$, forms an orthogonal basis
of $\mathcal{H}_{k}\left( \mathbb{R}^{d}\right) $.
\end{theorem}

 Using the abreviation $\alpha _{s}:=s+\left( d-3\right) /2$, for $k\geq 1$ identity (\ref{eqintFF}) becomes 
\begin{equation*}
Y_{k,\left( s,l\right) }^{d}\left( x\right) 
=
Y_{s,l}^{d-1}\left(
x_{1},\dots ,x_{d-1}\right) \cdot \left\vert x\right\vert
^{k-s}P_{k-s}^{\left( \alpha _{s},\alpha _{s}\right) }\left( \frac{x_{d}}{
\left\vert x\right\vert }\right).
\end{equation*}
Note  that the same factor $Y_{s,l}^{d-1}\left(
x_{1},\dots,x_{d-1}\right)$ appears in the preceding expression for different values of $k$, when $\left\vert x\right\vert =1$. 
The recurrence relation (\ref{eqrec2}) together with (\ref{eqintFF}) imply that for $k\ge 2$,
\begin{equation*}
x_{d}^{2}Y_{k,\left( s,l\right) }^{d}\left( x\right) =\widetilde{a}
_{k-s}\left( \alpha _{s}\right) Y_{k+2,\left( s,l\right) }^{d}\left(
x\right) +\widetilde{b}_{k-s}\left( \alpha _{s}\right) Y_{k,\left(
s,l\right) }^{d}\left( x\right) +\widetilde{g}_{k-s}\left( \alpha
_{s}\right) Y_{k-2,\left( s,l\right) }^{d}\left( x\right).
\end{equation*}
\color{red}  \color{black}
  Define the normalized spherical
harmonic by
\begin{equation*}
\widetilde{Y}_{k,\left( s,l\right) }^{d}=d_{k,\left( s,l\right) }\cdot
Y_{k,\left( s,l\right) }^{d}\text{ where }d_{k,\left( s,l\right)
}:=\left\Vert Y_{k,\left( s,l\right) }^{d}\right\Vert _{\mathbb{S}
^{d-1}}^{-1}.
\end{equation*}
It follows that 
\begin{eqnarray}
x_{d}^{2}\widetilde{Y}_{k,\left( s,l\right) }^{d}\left( x\right) &
=
&
\widetilde{a}_{k-s}\left( \alpha _{s}\right) \frac{d_{k,\left( s,l\right) }}{
d_{k+2,\left( s,l\right) }} 
\widetilde{Y}_{k+2,\left( s,l\right) }^{d}\left(
x\right) +\widetilde{b}_{k-s}\left( \alpha _{s}\right) \widetilde{Y}
_{k,\left( s,l\right) }^{d}\left( x\right)  \label{eqrell} \\
&&
+ \  \widetilde{g}_{k-s}\left( \alpha _{s}\right) 
\frac{d_{k,\left( s,l\right) }}{d_{k-2,\left( s,l\right) }}
\widetilde{Y}_{k-2,\left( s,l\right)}^{d}\left( x\right)  
\label{eqrell2}.
\end{eqnarray}

The following result, cf. \cite[Lemma 12]{RendAl22}, shows that to prove Theorem \ref{arbdimB} we only need
to  consider homogeneous polynomials $f_{2m}$ of even degree.

\begin{lemma}
\label{Lemevenodd} Let $P_{2k}$ be a homogeneous polynomial of degree $2k>0$.
 Suppose that for each $m\in \mathbb{N}$, there exists a constant $
C_{2m}>0$ such that for all homogeneous polynomials $f_{2m}$ of degree $2m$, we have
\begin{equation*}
\left\langle P_{2k}f_{2m},f_{2m}\right\rangle _{L^{2}(\mathbb{S}^{d-1})}\geq
C_{2m}\left\langle f_{2m},f_{2m}\right\rangle _{L^{2}(\mathbb{S}^{d-1})}.
\end{equation*}
Then for all homogeneous polynomials $f_{2m+1}$ of degree $2m+1$, 
\begin{equation*}
\left\langle P_{2k}f_{2m+1},f_{2m+1}\right\rangle _{L^{2}(\mathbb{S}
^{d-1})}\geq C_{2m+2}\left\langle f_{2m+1},f_{2m+1}\right\rangle _{L^{2}(
\mathbb{S}^{d-1})}.
\end{equation*}

\end{lemma}

Let us recall Theorem \ref{arbdimB}.

\begin{theorem} \label{arbdimBB}
	For all homogeneous polynomials $
	f_{m}$ of degree $m$ in $d$ variables, 
	and each $j = 1, \dots , d$, we have
	\begin{equation*}
		\left\langle x_{j}^{2}f_{m},f_{m}\right\rangle _{L^{2}\left( \mathbb{S}%
			^{d-1}\right) }
		\geq
		\frac{\pi ^{2}}{4\left( m+ 2 d + 1 \right)^{2}}
		\left
		\langle
		f_{m},f_{m}\right\rangle _{L^{2}\left( \mathbb{S}^{d-1}\right) }.
	\end{equation*}
\end{theorem}

\begin{proof} Without loss of generality we set $j = d$.
  By the preceding Lemma it suffices
to show that for all even indices we can take
$$C_{2m} := \frac{\pi^{2}}{4\left( 2m+ 2 d\right) ^{2}}.
$$

 By the Gauss
decomposition (cf. Theorem 5.5 in \cite{ABR92} or Theorem 5.7 in the 2001 edition) there exist homogeneous
harmonic polynomials $h_{2k}$ of degree $2k$ for $k=0,\dots ,m$, such that 
\begin{equation}
f_{2m} (x) =\sum_{k=0}^{m}  \left\vert x\right\vert ^{2m-2k} h_{2k}(x)\text{.}
\label{eqAlmansi1}
\end{equation}
 Each harmonic polynomial $h_{2k}$
of degree $2k$ for $k=0,\dots ,m$ is a linear combination of the orthonormal
spherical harmonics $\widetilde{Y}_{2k,\left( s,l\right) }^{d}$ with $
s=0,1,2,\dots ,2k$ and $l=1,\dots ,a_{s}^{(d-1)}$. Thus, we can write 
\begin{equation*}
h_{2k}=\sum_{s=0}^{2k}\sum_{l=1}^{a_{s}^{\left( d-1\right) }}c_{k,s,l}^{d}
\widetilde{Y}_{2k,\left( s,l\right) }^{d}
\end{equation*}
for suitable complex coefficients $c_{k,s,l}^{d}.$ \color{red} \color{black}

Recall that $\lfloor x \rfloor$ denotes the integer part of $x$. By reordering it is easy
to see that 
\begin{equation*}
f_{2m} (x) 
=
\sum_{k=0}^{m} \left\vert x\right\vert
^{2m-2k} h_{2k} (x)
=
\sum_{s=0}^{2m}\sum_{l=1}^{a_{s}^{\left( d-1\right) }}\sum_{k=\left\lfloor
\frac{s+1}{2}\right\rfloor }^{m}c_{k,s,l}^d\left\vert x\right\vert ^{2m-2k}%
\widetilde{Y}_{2k,\left( s,l\right) }^{d}  (x).
\end{equation*}%
By orthogonality 
\begin{equation}
\left\langle f_{2m},f_{2m}\right\rangle _{L^2(\mathbb{S}^{d-1})}=\sum_{s=0}^{2m}%
\sum_{l=1}^{a_{s}^{\left( d-1\right) }}\sum_{k=\left\lfloor \frac{s+1}{2}\right\rfloor
}^{m}\left\vert c_{k,s,l}^d\right\vert ^{2}  \label{eqF1}
\end{equation}
(since we are integrating over the unit sphere).
Now $\left\langle x_{d}^{2}f_{2m},f_{2m}\right\rangle _{L^2(\mathbb{S}^{d-1})}$
is equal to 
\begin{equation*}
\sum_{s_{1}=0}^{2m}\sum_{l_{1}=1}^{a_{s_1}^{\left( d-1\right) }}\sum_{k_{1}= 
\left\lfloor \frac{s_1+1}{2}\right\rfloor }^{m}\sum_{s_{2}=0}^{2m}\sum_{l_{2}=1}^{a_{s_2}^{
\left( d-1\right) }}\sum_{k_{2}=\left\lfloor \frac{s_2+1}{2}\right\rfloor
}^{m}c_{k_{1},s_{1},l_{1}}^d\overline{c_{k_{2},s_{2},l_{2}}^d}\left\langle
x_{d}^{2}\widetilde{Y}^d_{2k_{1},\left( s_{1},l_{1}\right) },\widetilde{Y}^d
_{2k_{2},\left( s_{2},l_{2}\right) }\right\rangle _{L^2(\mathbb{S}^{d-1})},
\end{equation*}
so by orthogonality (using the relation (\ref{eqrell}-\ref{eqrell2}) for $
x_d^2 \widetilde{Y}^d_{2k,\left( s,l\right) }$)
we obtain    
\begin{equation}
\left\langle x_{d}^{2}f_{2m},f_{2m}\right\rangle _{L^2(\mathbb{S}^{d-1})}=\sum_{s=0}^{2m}\sum_{l=1}^{a_{s}^{\left( d-1\right) }}\sum_{k_{1}=
\left\lfloor \frac{s+1}{2}\right\rfloor }^{m}\sum_{k_{2}=\left\lfloor \frac{s+1}{2}\right\rfloor
}^{m}c^d_{k_{1},s,l}\overline{c^d_{k_{2},s,l}}A_{k_{1},k_{2}}\left( s,l\right)
\label{eqF2}
\end{equation}%
where 
\begin{equation*}
A_{k_{1},k_{2}}\left( s,l\right) :=\left\langle x_{d}^{2}\widetilde{Y}^d
_{2k_{1},\left( s,l\right) },\widetilde{Y}^d_{2k_{2},\left( s,l\right)
}\right\rangle_{L^2(\mathbb{S}^{d-1})}.
\end{equation*}
In view of (\ref{eqF1}) and (\ref{eqF2}), it suffices to show that for each $
s=0,\dots , 2m$ and each $l=1,\dots ,a_{s}^{\left( d-1\right) }$,
\begin{equation}
\sum_{k_{1}=\left\lfloor \frac{s+1}{2}\right\rfloor }^{m}\sum_{k_{2}=\left\lfloor \frac{s+1}{2}%
\right\rfloor }^{m}c^d_{k_{1},s,l}\overline{c^d_{k_{2},s,l}}A_{k_{1},k_{2 }}\left(
s,l\right) 
\geq
 \frac{\pi ^{2}}{4\left( 2m + 2d  \right) ^{2}}\sum_{k=\left\lfloor 
\frac{s+1}{2}\right\rfloor }^{m}\left\vert c^d_{k,s,l}\right\vert ^{2}.  \label{eqF3}
\end{equation}
 The integer $s$ is either even or odd,
so we can write $s=2\sigma +\delta $ with $\delta \in \left\{ 0,1\right\}.$ 
Then 
\begin{equation*}
\left\lfloor \frac{s+1}{2}\right\rfloor =\left\lfloor \frac{2\sigma +1+\delta }{2}\right\rfloor
=\sigma +\delta .
\end{equation*}%
Let us define the symmetric $\left( m-\sigma -\delta +1\right) \times \left(
m-\sigma -\delta +1\right) $-matrix 
\begin{equation*}
A_{m}\left( s,l\right) :=\left( A_{k_{1},k_{2}}\left( s,l\right) \right)
_{k_{1},k_{2}=\sigma +\delta ,\dots , m}.
\end{equation*}
Writing $c^t = \left(c^d_{\left\lfloor \frac{s +1}{2}\right\rfloor, s,l}, \dots,  c^d_{m,s,l}\right)$, it is clear that the left hand side of (\ref{eqF3}) equals 
$$
c^t \ A_m(s,l) \ \overline{c},
$$
so it is enough to show that  the  smallest eigenvalue $\lambda _{m, s,l}^{\ast }$  of
 $A_m \left( s,l\right) $
satisfies the inequality 
$$
\lambda _{m, s,l}^{\ast }
\ge
\pi ^{2}/4\left( 2m+2d\right) ^{2}.
$$
Since $A_{m}\left( 2\sigma +1,l\right) $ is the submatrix of $A_{m}\left(
2\sigma ,l\right)$ obtained by deleting the first row and the first column, its lowest eigenvalue is at least as large as $\lambda _{m, 2\sigma,l}^{\ast }$. Thus, it suffices to prove (\ref{eqF3}) for $
A_{m}\left( 2\sigma ,l\right) $, or equivalently, for  $\delta =0.$

We now use the recurrence relation (\ref{eqrell} - \ref{eqrell2}) (with $2 k$ 
instead of  $k$)  
\begin{eqnarray*}
x_{d}^{2}\widetilde{Y}_{2k,\left( s,l\right) }^{d}\left( x\right) &=&%
\widetilde{a}_{2k-s}\left( \alpha _{s}\right) \frac{d_{2k,\left( s,l\right) }%
}{d_{2k+2,\left( s,l\right) }}\widetilde{Y}_{2k+2,\left( s,l\right)
}^{d}\left( x\right) +\widetilde{b}_{2k-s}\left( \alpha _{s}\right) 
\widetilde{Y}_{2k,\left( s,l\right) }^{d}\left( x\right) \\
&&+\widetilde{g}_{2k-s}\left( \alpha _{s}\right) \frac{d_{2k,\left(
s,l\right) }}{d_{2k-2,\left( s,l\right) }}\widetilde{Y}_{2k-2,\left(
s,l\right) }^{d}\left( x\right) .
\end{eqnarray*}

From the orthonormality relations we see that 
\begin{equation*}
A_{k,k}\left( s,l\right) =\left\langle x_{d}^{2}\widetilde{Y}^d_{2k,\left(
s,l\right) },\widetilde{Y}^d_{2k,\left( s,l\right) }\right\rangle_{L^2(\mathbb{S}^{d-1})}
=\widetilde{b}_{2k-s}\left( \alpha _{s}\right), 
\end{equation*}
\begin{equation*}
A_{k,k+1}\left( s,l\right) =\left\langle x_{d}^{2}\widetilde{Y}^d_{2k,\left(
s,l\right) },\widetilde{Y}^d_{2k+2,\left( s,l\right) }\right\rangle _{L^2(\mathbb{S}^{d-1})}
=
\widetilde{a}_{2k-s}\left( \alpha _{s}\right) \frac{d_{2k,\left(
s,l\right) }}{d_{2k+2,\left( s,l\right) }},
\end{equation*}
and 
\begin{equation*}
A_{k,k-1}\left( s,l\right) =\left\langle x_{d}^{2}\widetilde{Y}^d_{2k,\left(
s,l\right) },\widetilde{Y}^d_{2k-2,\left( s,l\right) }\right\rangle _{L^2(\mathbb{S}^{d-1})}
=
\widetilde{g}_{2k-s}\left( \alpha _{s}\right) \frac{d_{2k,\left(
s,l\right) }}{d_{2k-2,\left( s,l\right) }},
\end{equation*}
for $k=\sigma ,\dots , m$ and $s=2\sigma$, with all the other entries being equal to zero.  It follows that $
A_{m}\left( s,l\right) $ is an $\left( m-\sigma +1\right) \times \left(
m-\sigma +1\right) -$matrix of the form 
\begin{equation*}
\left( 
\begin{array}{cccc}
\widetilde{b}_{0} & \widetilde{a}_{0}\frac{d_{2\sigma ,\left( s,l\right) }}{%
d_{2\sigma +2,\left( s,l\right) }} &  &  \\ 
\widetilde{g}_{2}\frac{d_{2\sigma +2,\left( s,l\right) }}{d_{2\sigma ,\left(
s,l\right) }} & \widetilde{b}_{2} & \ddots &  \\ 
& \ddots & \ddots & \widetilde{a}_{2m - 2-2\sigma }\frac{d_{2m-2,\left( s,l\right) }}{d_{2m,\left( s,l\right)
}} \\ 
&  & \widetilde{g}_{2m-2\sigma }\frac{
d_{2m,\left( s,l\right) }}{d_{2m-2,\left( s,l\right) }} & \widetilde{b}_{2m-2\sigma }
\end{array}%
\right),
\end{equation*}
  We recall some known facts about tridiagonal $n\times n$-matrices of the form 
\begin{equation*}
J_{n}=\left( 
\begin{array}{cccc}
\beta _{0} & \alpha _{0} &  &  \\ 
\gamma _{1} & \beta _{1} & \ddots &  \\ 
& \ddots & \ddots & \alpha _{n-2} \\ 
&  & \gamma _{n-1} & \beta _{n-1}%
\end{array}
\right). 
\end{equation*}%
where all the entries $\alpha _{k}$ and $\gamma _{k}$ are non-zero. In order to compute
the characteristic polynomial $\det \left( J_{n}-\lambda I_n \right)$
 of $J_{n}$,  we define  a sequence of polynomials $p_{0}\left( x\right) :=1$, 
 $
 p_{1}\left( x\right) :=\left( x-\beta _{0}\right) /\alpha _{0}$, 
 and for $0<k<n$,   we use the recurrence relation  
 \begin{equation}
 	\label{genrec}
 	xp_{k}\left( x\right) =\gamma _{k}p_{k-1}\left( x\right) +\beta
 	_{k}p_{k}\left( x\right) +\alpha _{k}p_{k+1}\left( x\right).
 \end{equation}
 Then an inductive argument for $n\ge 1$, together with the  expansion of the determinant  along the last column, yields that $p_{n}\left( x\right) $ is a constant multiple
of the characteristic polynomial of $J_{n}$, 
and more specifically,  that
\begin{equation*}
\det \left( J_{n}-\lambda I_n   \right) =\left( -1\right) ^{n}\alpha
_{0}\dots \alpha _{n-2}p_{n}\left( \lambda \right) .
\end{equation*}%
Thus the eigenvalues of $J_{n}$ are the zeros of the polynomials $
p_{n}\left( x\right)$ defined by (\ref{genrec}).  Next we consider the special form $
A_{m}\left( s,l\right) $ of  $J_n$ that appears in our argument, with an additional simplification.  Replacing $p_{k}$ with $
q_{k}:=d_{k}p_{k}$ (where $d_{k}\neq 0)$ in
(\ref{genrec}) we obtain the recurrence relation 
\begin{equation*}
xq_{k}\left( x\right) =\alpha _{k}\frac{d_{k}}{d_{k+1}}q_{k+1}\left(
x\right) +\beta _{k}q_{k}\left( x\right) +\gamma _{k}\frac{d_{k}}{d_{k-1}}
q_{k-1}\left( x\right),
\end{equation*}
with the initial conditions transformed to $q_{0}\left( x\right) =d_{0}$
and $q_{1}\left( x\right) =d_{1}\left( x-\beta _{0}\right) /\alpha _{0}.$

It follows that the matrix $A_{m}\left( s,l\right) $, for $s=2\sigma $, has the
same eigenvalues as the matrix
\begin{equation*}
\left( 
\begin{array}{cccc}
\widetilde{b}_{0} & \widetilde{a}_{0} &  &  \\ 
\widetilde{g}_{2} & \widetilde{b}_{2} & \ddots  &  \\ 
& \ddots  & \ddots  & \widetilde{a}_{2m- 2 - 2\sigma } \\ 
&  & \widetilde{g}_{2m-2\sigma } & \widetilde{b}_{2m-2\sigma }
\end{array}
\right). 
\end{equation*}%
Next we replace the even indices $2n$ appearing
in the preceding matrix with $u = 2n$, 
and we recognize it as the Jacobi matrix
whose eigenvalues are computed via the polynomials appearing in the recurrence relation (\ref{eqrec3}). 
Denote  by  $y_u$ the first positive zero  of $r_u (y) =
P_{2n}^{\left( \alpha ,\alpha \right) } \left( \sqrt{y}\right),$ 
with $n =  m-\sigma +1$ and $\alpha
=s+\left( d-3\right) /2.$  Then  the smallest eigenvalue $\lambda _{m, s,l}^{\ast }$  of $
A_{m}\left( s,l\right) $ is $y_u$.  
Note that $\sigma $ ranges from $
0$ to $m$, so $n =m-\sigma +1$ ranges from $1$ to $m+1.$ Furthermore,  
\begin{equation*}
\alpha =s+\left( d-3\right) /2=2\left( m+1-n\right) +\left( d-3\right) /2.
\end{equation*}
By Theorem \ref{Elbert},
\begin{equation*}
\left( \alpha +\frac{1}{2}+n\right) \left( n+2\right) y_u\geq \frac{\pi ^{2}}{16}.
\end{equation*}
For $n=m-\sigma +1$ we estimate     
\begin{eqnarray*}
\left( \alpha +\frac{1}{2}+n \right) \left( n+2\right)  &\leq
&\max_{n=1,\dots , m+1}\left( 2\left( m+1-n\right) +\left( d-3\right) /2 + 1/2 +    n \right)
\cdot \left( n+2\right)  \\
&\leq & \left( m + d \right)^{2},
\end{eqnarray*}
using the fact that $d \ge 2$, which in particular entails that for every $\sigma \ge 0$ we have $3 \sigma - \sigma^2 - (d \sigma)/2 \le 1$. 

It follows that  
\begin{equation*}
\lambda _{m, s,l}^{\ast } 
=
 y_{u} 
\geq 
\frac{\pi ^{2}}{16\left( m+ d\right)^{2}}.
\end{equation*}
\end{proof}

\section{Proof of Theorem \ref{arbdimParCyl}}

The
techniques of proof used here are based on the Fischer decomposition of polynomials, cf. \cite{RendAl22},  a framework that allows a unified treatment of all nonhyperbolic quadratic hypersurfaces.
We use the special case of \cite[Theorem 1]{RendAl22} where  $\alpha = 2$ and $k = 1$. 

\begin{theorem}
\label{ThmMain1} Let $P_{2}$ be a homogeneous polynomial of degree $2$
such that there exist $C>0$ and $D>0$ with 
\begin{equation}
\left\langle P_{2}f_{m},f_{m}\right\rangle _{L^{2}(\mathbb{S}^{d-1})}\geq 
\frac{1}{C\left( m+D\right) ^{2}}\left\langle f_{m},f_{m}\right\rangle
_{L^2(\mathbb{S}^{d-1})}  \label{eqmainassumption}
\end{equation}
for all homogeneous polynomials $f_{m}$ of degree $m.$  Let 
$q = P_2 + P_1 + P_0$, where $P_i$ denotes the homogeneous part of $q$ of degree $i$. Set 
$\beta =1$ if $P_1 \ne 0$ and $\beta = 0$ otherwise. 
 Then for every entire function $f$ of order $\rho <\left(
2-\beta \right) /2 $, there exist entire functions $s$ and $r$ of
order $\leq \rho $ such that 
\begin{equation*}
f= q \cdot s+r\text{ and }\Delta r=0.
\end{equation*}
\end{theorem}

Now Theorem \ref{arbdimB} allows us to apply the preceding result to any nonhyperbolic polynomial $q$, so
 Theorem \ref{arbdimParCyl} immediately follows, since clearly $r$ is harmonic and satisfies
 $r = f$ on $\{q = 0\}$.
 





\smallskip \noindent \textbf{Data availability} Data sharing is not
applicable to this article since no data sets were generated or analyzed.

\smallskip \noindent \textbf{Declarations}

\smallskip \noindent \textbf{Conflict of interest} The authors declare that
they have no competing interests.

\smallskip \noindent \textbf{Acknowlegement}  \thanks{The first named author was partially supported by Grant PID2019-106870GB-I00 of the MICINN of Spain,  by ICMAT Severo Ochoa project
	CEX2019-000904-S (MICINN), and by
	the Madrid Government (Comunidad de Madrid - Spain)  V PRICIT (Regional Programme of Research and Technological Innovation), 2022-2024.}

\end{document}